\documentclass[11pt]{amsart}

\usepackage{amsmath}
\usepackage[english]{babel}
\usepackage{amssymb}
\usepackage{enumitem}
\usepackage[initials]{amsrefs}
\usepackage[all]{xy}
\usepackage{bbm}
\usepackage{ulem}
\usepackage{float}

\usepackage[pdftex,colorlinks,citecolor=blue]{hyperref}

\usepackage{fancyhdr}

\usepackage{tikz}

\usepackage{amscd}
\usepackage{amsfonts}
\usepackage{mathrsfs}
\usepackage{setspace}
\usepackage{version}
\usepackage{mathtools}

\SelectTips{cm}{}
\usepackage[all]{hypcap}

\newcommand{\bbN}{{\mathbb N}}
\newcommand{\bbQ}{{\mathbb Q}}

\newcommand{\bbR}{{\mathbb R}}
\newcommand{\bbZ}{{\mathbb Z}}

\newcommand{\Ker}{\operatorname{Ker}}

\newcommand{\half}{\frac{1}{2}}

\newcommand{\efface}[1]{}

\newtheorem{mthm}{Theorem}
\newtheorem*{conj}{Conjecture}

\newtheorem{theorem}{Theorem}[section]
\newtheorem{lemma}[theorem]{Lemma}

\newtheorem{cor}[theorem]{Corollary}

\newtheorem{prop}[theorem]{Proposition}

\theoremstyle{definition}

\newtheorem{defn}[theorem]{Definition}

\newtheorem{example}[theorem]{Example}

\newtheorem{remark}[theorem]{Remark}

\numberwithin{equation}{section}

\begin{document}
\title[On the horofunction boundary of discrete Heisenberg group]{On the horofunction boundary of discrete Heisenberg group}

\author{Uri Bader}
\author{Vladimir Finkelshtein}

\address{Uri Bader, Faculty of Mathematics and Computer Science,
The Weizmann Institute of Science, 234 Herzl Street, Rehovot 7610001, Israel}
\email{uribader@gmail.com}
\address{Vladimir Finkelshtein, Mathematisches Institut, Georg-August-Universit\"{a}t G\"{o}ttingen, Bunsenstra\ss e 3-5, 37073 G\"{o}ttingen, Germany}
\email{filyok@gmail.com}

\begin{abstract}
 We consider finitely generated group endowed with a word metric. The group acts on itself by isometries, which induces an action on its horofunction boundary. The conjecture is that nilpotent groups act trivially on their reduced boundary. We will show this for the Heisenberg group. The main tool will be a discrete version of the isoperimetric inequality.
\end{abstract}

\maketitle



\section{Introduction}

Every metric space embeds in the space of continuous functions on it,
and its image there, modulo the constant functions is precompact.
The functions in the closure are denoted horofunctions, the closure of the image is denoted  the horofunction compactification
and the boundary is denoted the horofunction boundary or the horoboundary. This notion is due to Gromov \cite{G}.
The horoboundary carries a natural equivalence relation.
The corresponding quotient space is called the {\bf reduced horoboundary}.

Given a group with a specified set of generators,
one obtains a metric space by considering the corresponding word metric on the group,
and thus one gets corresponding horoboundary and reduced horoboundary.
The group acts naturally on those spaces.
Both of those spaces might depend on the choice of generators,
but in some cases topological and dynamical properties of the action do not.

A well-known example is given by hyperbolic groups,
for which the reduced horoboundary coincides with the Gromov boundary
and, in particular, does not depend on choice of generators
(while the horoboundary does).
Hyperbolic groups indeed provide a reach class of examples for groups
with non-trivial actions on their reduced horoboundaries.

On the other extreme, the reduced horoboundary of a finitely generated abelian group
depends on a choice of generators, but the boundary behavior is rather simple, as seen in the following theorem.

\begin{mthm}\label{thmA}
Given a finitely generated abelian group endowed with any finite set of generators,
the corresponding reduced horoboundary is finite and
the group action on it is trivial.
\end{mthm}

More generally, we conjecture the following.

\begin{conj}
Given a finitely generated nipotent group endowed with any finite set of generators,
the action of the group on its reduced horoboundary is trivial.
\end{conj}

The purpose of this paper is to establish this conjecture for the first non-trivial case.

\begin{mthm}\label{thmB}
Given any finite set of generators of the discrete Heisenberg group,
the action of the group on the corresponding reduced horoboundary is trivial.
\end{mthm}

The action of the Heisenberg group on its horoboundary was previously studied by Walsh in \cite{W}, where he established the existence of finite orbits.

We will prove the theorem above by introducing a new property: {\bf property EH},
which implies the triviality of the action of a group on its reduced horoboundary. 
Establishing property EH for the Heisenberg group will lead us to consider the norm function of the group (see \cite{B} for explicit description of this norm with standard generators) 
and, in particular, to prove a discrete version of the planar isoperimetric inequality, which we believe carries some independent interest.

Section \ref{sec2} below will be devoted to setting our notation and framework,
and in particular, for discussing property EH and its relevance to Theorems \ref{thmA} and \ref{thmB}.
We will discuss abelian groups and prove Theorem \ref{thmA} in section \ref{sec3}.
In section \ref{sec4} we will prove our discrete isoperimetric inequality.
In section \ref{sec5} we will discuss the norm function on the Heisenberg group
and prove Theorem \ref{thmB}.

\section{Reduced Horoboundaries and Property EH}\label{sec2}

Let $(X,d)$ be a proper metric space.
Endow $C(X)$ by the Frechet structure of uniform convergence on compact sets.
We denote by $C^0(X)$ the quotient Frechet space obtained by $C(X)$ when moding up the one dimensional subspace of constant functions.
We get a natural map:
\[ X\hookrightarrow C(X) \to C^0(X), \quad x \mapsto d(x,\cdot) \mapsto [d(x,\cdot)]. \]
It is trivial to check that the composition map is injective
(for this, it is enough to consider two point sets),
that it is a homeomorphism on the image (for $X$ proper)
and that the image is precompact (by Arzela-Ascoli theorem).
We denote the closure of the image of $X$ in $C^0(X)$ by $\overline{(X,d)}$ and, 
upon identifying $X$ with its image, we set $\partial (X,d)=\overline{(X,d)}-X$.
These are the horocompactification and the {\bf horoboundary} of $X$.

Consider the space $C_b(X)<C(X)$, consisting of all bounded continuous functions.  Let $C^r(X)=C(X)/C_b(X)$ be the quotient space.
The {\bf reduced horoboundary} of $X$, denoted by $\partial^r (X,d)$, is the image of $\partial (X,d)$ in $C^r(X)$.

For a finitely generated group $G$ with a finite symmetric set of generators $S$,
we denote the $S$-word metric on $G$ by $d_S$ and the corresponding norm on $G$ by $|\cdot|_S$.
We denote $\partial (G,S)$ and $\partial^r (G,S)$ for $\partial (G,d_S)$ and $\partial^r (G,d_S)$.
When the set $S$ is understood we simply denote the norm by $| \cdot |$ and the boundaries by $\partial G$ and $\partial^r G$.

\begin{defn}
Given a group $G$,
a finite symmetric set of generators $S\subset G$ is said to satisfy the {\bf property  EH} if
there exists a constant (called the EH constant of $S$) $D>0$ such that for every $g_0\in G'$ (the commutator group)
there exists $n\in\bbN$ satisfying
\[ \mbox{for all } y\in G, \quad\quad |y|_S>n \quad \Rightarrow \quad \big| |g_0y|_S-|y|_S \big|\leq D. \]
The group $G$ itself is said to satisfy EH if every finite symmetric set of generators of it satisfies EH.
\end{defn}

\begin{prop} \label{EH}
Let $G$ be a finitely generated group and let $S\subset G$ be a finite symmetric generating set satisfying EH.
Then the action of $G$ on $\partial^r(G,S)$ is trivial.
\end{prop}

\begin{proof}
Pick $\phi\in C(G)$ for which 
$[\phi] \in \partial G$.
Fix $g\in G$.
We need to show that $g\phi-\phi \in C_b(G)$.
We will show that for every $x\in G$,
\[ |g\phi(x)-\phi(x)| \leq D+|g|, \]
where $D$ is the EH constant of $S$.
Fix $x\in G$.
Consider the elements $[x^{-1},g], [x^{-1},g^{-1}]\in G'$.
Then
there exists $n\in\bbN$ such that for every $y\in G$ with $|y|>n$,
\begin{equation} \label{eq:y}
\left| |[x^{-1},g]y|-|y| \right|, \left| |[x^{-1},g^{-1}]y|-|y| \right|\leq D.
\end{equation}
Let $w\in G$ be an element with $|w|>n+2|g|+|x|$ such that
\begin{equation} \label{eq:phi=w}
|d(w,g^{-1}x)-d(w,x)| = |\phi(g^{-1}x)-\phi(x)|.
\end{equation}
Then
\[ d(gw,x)=|x^{-1}gw|=|[x^{-1},g]gx^{-1}w|, \]
and since $|gx^{-1}w|>n+|g|\geq n$ we get by substituting $y=gx^{-1}w$ in equation ~\ref{eq:y}
\begin{equation} \label{eq:leq}
d(gw,x) \leq |gx^{-1}w|+D \leq |x^{-1}w|+|g|+D= d(w,x) +|g|+D.
\end{equation}
On the other hand,
\[ d(w,x)=|x^{-1}w|=|[x^{-1},g^{-1}]g^{-1}x^{-1}gw|, \]
and since $|g^{-1}x^{-1}gw|>n$ we get by substituting $y=g^{-1}x^{-1}gw$ in equation ~\ref{eq:y}
\begin{equation} \label{eq:geq}
d(w,x) \leq |g^{-1}x^{-1}gw| +D \leq |x^{-1}gw|+|g| +D = d(gw,x)+|g|+D.
\end{equation}
Equations ~\ref{eq:leq},~\ref{eq:geq} together with equation ~\ref{eq:phi=w} give the desired inequality,
\[ |g\phi(x)-\phi(x)|=|\phi(g^{-1}x)-\phi(x)|=|d(gw,x)-d(w,x)| \leq D+|g|. \]
\end{proof}

\section{Abelian Groups - Proof of Theorem A.}\label{sec3}

In this section we consider finitely
generated abelian groups
and discuss their horoboundaries and reduced horoboundaries.
This question was studied by Rieffel \cite{R} and Develin \cite{D} in different generality.

Observe that every finitely generated abelian group is trivially EH, hence
the second part of Theorem \ref{thmA} follows immediately by
Proposition~\ref{EH}.
We are left to show that the reduced horoboundary is finite.
This is a consequence of the more general Proposition~\ref{abelian} below. To motivate the statement we first conisder a simple example.

\begin{example} \label{G=Z}
Let $G=\bbZ$ with the generating set $S=\{-1,+1\}$.
The horoboundary consists of (the classes of) the functions $\{x,-x\}$.
The map to the reduced horoboundary is a bijection.

Consider now the generating set $T=\{\pm1,\pm10\}$ for $G$.
The reduced horoboundary still consists of two points
(the classes of the functions $\pm\frac{1}{10}x$),
but the horoboundary consists of 20 points and the map is 10 to 1.
The horofunctions are limits of sequences of the distance functions from the points
\[ 10n, 10n+1,\ldots, 10n+9 \quad \mbox{and} \quad -10n, -10n+1,\ldots, -10n+9. \]
\end{example}

Note that the fibers of the map
$\partial G\to \partial^r G$
are subsets of cosets of $C_b(G)$, hence carry natural metrics.
In Example~\ref{G=Z}, both fibers of
$\partial (G,T)\to \partial^r (G,T)$
are isomorphic to the metric space $(\bbZ/10\bbZ,d_{\{\pm1\}})$.
See \cite{D} for more examples.

\begin{defn}
Let $G$ be a finitely generated abelian group and $S\subset G$ a finite symmetric generating set.
A nonempty subset $F\subset S$ is called a {\bf face} of $S$
if the following property holds:
for every $|S|$-tuple and $|F|$-tuple of non-negative integers $(\alpha_s)_{s\in S}$,
and $(\beta_f)_{f\in F}$,
satisfying
\[ \sum_{s\in S} \alpha_s = \sum_{f\in F} \beta_f \quad \mbox{and} \quad \sum_{s\in S} \alpha_s\cdot s = \sum_{f\in F} \beta_f\cdot f\]
we have $\alpha_s=0$ for every $s\notin F$.
\end{defn}

The faces of $T$ in Example~\ref{G=Z} are the singletons $\{-10\}$ and $\{+10\}$. Note that in case of free abelian groups $\bbZ^n$, the faces of the generating set are the faces of the convex hull of the generators embedded in $\mathbb{R}^{n}$ intersected with $S$.

Recall that every $T_0$ finite topological space is nothing but a finite poset, upon setting for points $x$ and $y$,
\[ x\leq y \quad \Leftrightarrow \quad y\in \overline{\{x\}}. \]

\begin{prop} \label{abelian}
Let $G$ be a finitely generated abelian group and $S$ a finite symmetric generating set.
Then the set $\partial^r (G,S)$ is in one-to-one correspondence with the collection of faces of $S$.
In particular, $\partial^r (G,S)$ is a finite set.
Moreover, under this correspondence we have the following.
\begin{enumerate}
\item
For every face $F\subset S$, the corresponding fiber in $\partial (G,S)$ is isometric to the Cayley graph of $(G/\langle F \rangle,S\langle F \rangle)$.
\item
The quotient topology on $\partial^r (G,S)$ is $T_0$ and the correspondence with the set of faces is order preserving,
for the topology ordering on $\partial^r (G,S)$ and inclusion of faces.
\item
The simplicial complex of flags in the poset $\partial^r (G,S)$ is homemorphic to a sphere. Its dimension equals the rank of $G$ minus one.
\end{enumerate}
\end{prop}

\begin{proof}
First we remark that points along a geodesic ray always converge to a horofunction (see \cite[Section 1.2]{G} or \cite[Theorem 4.7]{R}).

The definition of the face implies that there exists $M>0$ such that for any geodesic $r$ there is a face $F\subseteq S$ such that number of letters in $r$  which are not from $F$ is bounded by $M$, in other words, up to finite translation, all the geodesics rays are given by using infinitely many letters from some face of $S$ and finitely many from others.  
Because the group is abelian, any two geodesic rays that use infinitely many letters from the same face and finitely many from other faces are equivalent, in the sense that they converge to the same limit point in the reduced horoboundary. Therefore, there are finitely many (up to this equivalence) geodesic rays, and any unbounded sequence of elements in $G$ lies, up to subsequence, on such a geodesic ray.

Given a geodesic ray $r$ one can find a minimal face (with respect to inclusion) which contains all the letters which appear infinitely often in $r$. Conversely, given a face $F\subseteq S$ one can build a geodesic ray using only letters from $F$ and using each one of them infinitely many times. This defines a bijection between the faces of $S$ and geodesic rays converging to distinct points on the reduced horoboundary. To see that the latter is true, let $r_1(t), r_2(t)$ be the geodesic rays corresponding to two different faces $F_1,F_2$, $\phi_1,\phi_2$
the limiting horofunctions, normalized such that $\phi_1(0)=\phi_2(0)=0$, where we write $0$ for the identity element of the group. Without loss of generality there exists $a\in F_1 \setminus F_2$. Clearly, $\phi_1(-na)=n$ for all $n\in \bbN$. To prove that $\phi_1 - \phi_2 \notin C_b(G)$, we argue that for any $C\geq 0$ we have $\phi_2(-na) \leq n-C$ for large enough values of $n$. This is easily verified after projecting to the torsion free part of $G$. For the same reason, if $F_1\not\subset F_2$, we have $[\phi_2] \notin \overline{ \{ [\phi_1] \}}$, implying that the quotient topology is $T_0$. 

The fibers of a point in the reduced horoboundary corresponding to a face $F$ are all translations of a geodesic ray, which uses each element in $F$ infinitely many times. These are exactly the elements of $G/\langle F \rangle$.

To see that the correspondence is order preserving, suppose $F_1 \subset F_2$. Let $\phi_1, \phi_2$ be the corresponding horofunctions, where $\phi_1$ is obtained as a limit along the sequence $n\sum_{f\in F_1} f$ and $\phi_2$ as a limit along the sequence $n\sum_{f\in F_2} f$ as $n\to \infty$. We will show that there exists a sequence of horofunctions $\psi_j\in \partial(G,S)$ such that $\psi_j-\phi_1 \in C_b(G)$ for all $j$ and $\psi_j\to \phi_2$ as $j\to \infty$. Indeed, one can take $\psi_j$ as a limit along $j\sum_{f\in F_2\setminus F_1}f + n\sum_{f\in F_1}f$ as $n\to \infty$. The above properties are clearly satisfied, and hence $[\phi_2] \in \overline{\{ [\phi_1] \}}$.

Let $G=\mathbb{Z}^n\times T$, where $T$ is the torsion part. We will show that the projection $\pi : G \to \mathbb{Z}^n$ defines an order preserving bijection of the faces, with respect to generating sets $S$ and $\pi(S)$. 

If $F \subset S$ is a face, suppose we have  $\sum_F \alpha_f \pi(f) = \sum_S \alpha_s \pi(s)$ and $\sum_F \alpha_f = \sum_S \alpha_s$. Write preimages of the first equation. If there is no equality in the preimage, then the difference between the sides is in the torsion part, hence by multiplying all the coefficients, one will obtain equality in the preimage. Hence, we can assume that these equalities hold in the preimage of $\pi$, thus $\alpha_s=0$ for $s\notin F$. We need to show that $s \notin F$ implies $\pi(s) \notin \pi(F)$(this will also imply injectivity). Suppose not, then for some $f\in F$, $f=s+t$ where $t$ is the torsion part. Then for some $\alpha \neq 0$, $\alpha t=0$, hence $\alpha s = \alpha f$, but since $s\notin F$, by definition of face, we must have $\alpha=0$, contradiction. Thus, $\pi$ maps faces to faces. Clearly, $\pi$ preserves inclusion. 

To see that the topology is $T_0$, note that we already showed that the closure of points corresponding to a face contains all maximal faces in which this face is contained. We are left to show that a singleton corresponding to a maximal face is closed. This would describe all closures of points, which will be different for different points.

To show surjectivity of $\pi$ let $\bar{F}\subset \pi(S)$ be a face, i.e. for any combination $\sum_{\bar{F}} \alpha_f \pi(f) = \sum_{\pi(S)} \alpha_s \pi(s)$ such that $\sum_{\bar{F}} \alpha_f = \sum_{\pi(S)} \alpha_s$, we have $\alpha_s=0$ for $\pi(s)\notin \bar{F}$. Let $F=\pi^{-1}(\bar{F})$. Need to show that $F\subset S$ is a face. For any combination $\sum_F \alpha_f f = \sum_{S} \alpha_s (s)$ such that $\sum_F \alpha_f = \sum_S \alpha_s$ the same holds after applying $\pi$, therefore, $\alpha_s=0$ for $\pi(s)\notin \bar{F}$. As we have already seen, that $s\notin F$ implies $\pi(s)\notin \pi(F)$, then $\alpha_s=0$ for $s \notin F$, and therefore the preimage of $\bar{F}$ is a face. 

Hence, the simplicial complex of flags in the poset $\partial^r (G,S)$ is homemorphic to one obtained from $\partial^r (\pi(G),\pi(S))$, where $\pi$ is the map to the torsion free component. For free abelian group $\mathbb{Z}^n$, the faces in our sense coincide with faces of convex hull of the generators embedded in $\mathbb{R}^n$, and the corespondence  preserves the order, hence the flag complex is homeomorphic to $(n-1)$-sphere, where $n$ is the rank of the group.

\end{proof}

\section{An Isoperimetric Inequality for $\mathbb{Z}^{2}$}\label{sec4}

In this section we consider an elementary geometric problem,
a discrete planar isoperimetric inequality,
which might be of an independent interest of the rest of the paper. We start by defining notions needed to state the discrete isoperimetric inequality. 

Fix a finite collection of vectors $V\subset \bbR^2$. Assume that $V=-V$.
A $V$-\textbf{polygon} (or simply, a polygon when $V$ is clear) is
a word in the kernel of the natural map $F_V\to \bbR^2$ where $F_V$ is the free group generated by $V$.
Put in another way, it is a word in $V$ which represents the trivial element in $\bbR^2$. 
We denote by $\mathcal{P}(V)$ (or $\mathcal{P}$ when $V$ is clear) the collection of all $V$-polygons. 

For a polygon $P=(u_1\ldots u_n)$, $u_i\in V$, (and $\sum u_i=0$),  set $l(P)=n$ and $a(P)=\half\sum_{i<j} \det(u_i,u_j)$. The geometric realization of $P$ is the polygon in $\bbR^2$ obtained by concatenating the vectors $u_i$ in this order.
The quantities $l(P)$ and $a(P)$ are the (combinatorial) \textbf{perimeter} and the signed (Euclidean) \textbf{area}, respectively, of the geometric realization of $P$. Indeed, for fixed $1\leq j \leq n$, the signed area of $\Delta_j$ in Figure \ref{fig:geom-real} is given by $\frac{1}{2}\sum _{i< j}\det( u_{i} ,u_{j})$. For the area of $P$, we sum over $j$, i.e.  $a(P)=\sum_j a(\Delta_j)$.

\begin{figure}[t]\label{fig:geom-real}
 
\tikzset{
pattern size/.store in=\mcSize, 
pattern size = 5pt,
pattern thickness/.store in=\mcThickness, 
pattern thickness = 0.3pt,
pattern radius/.store in=\mcRadius, 
pattern radius = 1pt}
\makeatletter
\pgfutil@ifundefined{pgf@pattern@name@_9ppp0o5ch}{
\makeatletter
\pgfdeclarepatternformonly[\mcRadius,\mcThickness,\mcSize]{_9ppp0o5ch}
{\pgfpoint{-0.5*\mcSize}{-0.5*\mcSize}}
{\pgfpoint{0.5*\mcSize}{0.5*\mcSize}}
{\pgfpoint{\mcSize}{\mcSize}}
{
\pgfsetcolor{\tikz@pattern@color}
\pgfsetlinewidth{\mcThickness}
\pgfpathcircle\pgfpointorigin{\mcRadius}
\pgfusepath{stroke}
}}
\makeatother

 
\tikzset{
pattern size/.store in=\mcSize, 
pattern size = 5pt,
pattern thickness/.store in=\mcThickness, 
pattern thickness = 0.3pt,
pattern radius/.store in=\mcRadius, 
pattern radius = 1pt}
\makeatletter
\pgfutil@ifundefined{pgf@pattern@name@_0o45n4tmv}{
\makeatletter
\pgfdeclarepatternformonly[\mcRadius,\mcThickness,\mcSize]{_0o45n4tmv}
{\pgfpoint{-0.5*\mcSize}{-0.5*\mcSize}}
{\pgfpoint{0.5*\mcSize}{0.5*\mcSize}}
{\pgfpoint{\mcSize}{\mcSize}}
{
\pgfsetcolor{\tikz@pattern@color}
\pgfsetlinewidth{\mcThickness}
\pgfpathcircle\pgfpointorigin{\mcRadius}
\pgfusepath{stroke}
}}
\makeatother
\tikzset{every picture/.style={line width=0.75pt}} 
\usetikzlibrary{patterns}

\begin{tikzpicture}[x=0.75pt,y=0.75pt,yscale=-.7,xscale=.7]

\draw   (369.29,70.79) -- (422.73,102.42) -- (396.11,153.01) -- (309.49,184.47) -- (228.09,173.1) -- (213.2,127.47) -- (276.04,81.94) -- cycle ;
\draw [pattern=_9ppp0o5ch,pattern size=6pt,pattern thickness=0.75pt,pattern radius=0.75pt, pattern color={rgb, 255:red, 0; green, 0; blue, 0}] [dash pattern={on 0.84pt off 2.51pt}]  (228.09,173.1) -- (367.67,71.96) ;
\draw [shift={(369.29,70.79)}, rotate = 504.07] [fill={rgb, 255:red, 0; green, 0; blue, 0 }  ][line width=0.75]  [draw opacity=0] (8.93,-4.29) -- (0,0) -- (8.93,4.29) -- cycle    ;

\draw [pattern=_0o45n4tmv,pattern size=6pt,pattern thickness=0.75pt,pattern radius=0.75pt, pattern color={rgb, 255:red, 0; green, 0; blue, 0}] [dash pattern={on 0.84pt off 2.51pt}]  (228.09,173.1) -- (276.04,81.94) ;

\draw    (286,80.67) -- (278.02,81.68) ;
\draw [shift={(276.04,81.94)}, rotate = 352.74] [color={rgb, 255:red, 0; green, 0; blue, 0 }  ][line width=0.75]    (10.93,-3.29) .. controls (6.95,-1.4) and (3.31,-0.3) .. (0,0) .. controls (3.31,0.3) and (6.95,1.4) .. (10.93,3.29)   ;

\draw (219,186) node  [align=left] {$\displaystyle 0$};
\draw (311,204) node  [align=left] {$\displaystyle u_{1}$};
\draw (405,171) node  [align=left] {$\displaystyle u_{1} +u_{2}$};
\draw (389,47) node  [align=left] {$\displaystyle \sum _{i< j} u_{i}$};
\draw (287,103) node  [align=left] {$\displaystyle \Delta_j $};
\draw (321,59) node  [align=left] {$\displaystyle u_{j}$};

\end{tikzpicture}
\caption{Geometric realization of $\displaystyle P.$}
\end{figure}
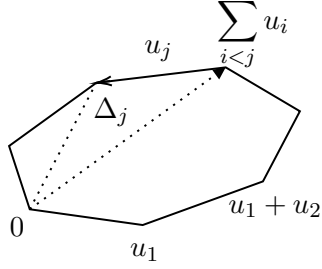

We also set
\[ \gamma(P):=a(P)/l(P)^2 \quad \mbox{and} \quad \gamma_V:=\sup \{\gamma(P)~|~P\in \mathcal{P}(V)\}. \]
The constant $\gamma_V$ is called the \textbf{isoperimetric constant} of $V$.

We define special families of polygons which will be useful in the proofs. 
A polygon of the form
$P=(u_1,\ldots,u_n,-u_1,\ldots,-u_n)$ is said to be   \textbf{symmetric}. If $P$ is a symmetric polygon, denote by 
$$\half  P:=(u_1,\ldots, u_n, -u_1-u_2-\ldots -u_n).$$
Note that for symmetric polygon $P$ of length $2n$ the area is
\begin{equation}\label{eq:areasymmetric}
    a(P)= 2a\left(\half P\right) = \sum_{i<j\leq n} \det(u_i,u_j)
\end{equation}  
We endow the set $V$ with the order $\prec$ induced from the order on the arguments of vectors in $\bbR^2$, where the arguments are seen as an interval $[0,2\pi)$.  A symmetric polygon $P=(u_1,\ldots,u_n, -u_1, \ldots, -u_n)$ is said to be \textbf{ordered} if up to cyclic permutation for all $i\leq j\leq n$ we have $u_i \prec u_j$. A polygon $P$ is ordered if and only if the geometric realization of $\half P$ is convex and the signed area $a(P)$ is non-negative. Note that in our definition the geometric realization of an ordered polygon $P$ is not necessarily convex itself, as the angle between $u_n$ and $-u_1$ can be larger than $\pi$.
We denote the set of all symmetric ordered polygons by $\mathcal{P}_{so}$.
 
 \begin{figure}\label{fig:symorder}
    
\tikzset{every picture/.style={line width=0.75pt}} 

\begin{tikzpicture}[x=0.75pt,y=0.75pt,yscale=-.7,xscale=.7]

\draw   (578,143) -- (540.12,196.03) -- (480.5,218) -- (434.06,196.03) -- (428,143) -- (465.88,89.97) -- (525.5,68) -- (571.94,89.97) -- cycle ;
\draw    (434.06,196.03) -- (571.94,89.97) ;

\draw   (172.5,116) -- (209.5,50) -- (255.94,71.97) -- (262,125) -- (201.5,134) -- (164.5,200) -- (118.06,178.03) -- (112,125) -- cycle ;

\draw (447,216) node  [align=left] {$\displaystyle u_{1}$};
\draw (557,69) node  [align=left] {\mbox{-}$\displaystyle u_{1}$};
\draw (590,114) node  [align=left] {$\displaystyle u_{n}$};
\draw (415,169) node  [align=left] {\mbox{-}$\displaystyle u_{n}$};
\draw (519,173) node  [align=left] {$\displaystyle \frac{1}{2} P$};
\draw (517,219) node  [align=left] {$\displaystyle u_{2}$};
\draw (487,69) node  [align=left] {\mbox{-}$\displaystyle u_{2}$};
\draw (134,204) node  [align=left] {$\displaystyle u_{1}$};
\draw (98,157) node  [align=left] {\mbox{-}$\displaystyle u_{n}$};
\draw (282,94) node  [align=left] {$\displaystyle u_{n}$};
\draw (250,55) node  [align=left] {\mbox{-}$\displaystyle u_{1}$};
\draw (240,140) node  [align=left] {$\displaystyle u_{2}$};
\draw (136,105) node  [align=left] {\mbox{-}$\displaystyle u_{2}$};

\end{tikzpicture}

     \caption{Symmetric unordered (left) and ordered (right) polygons.}
     
 \end{figure}
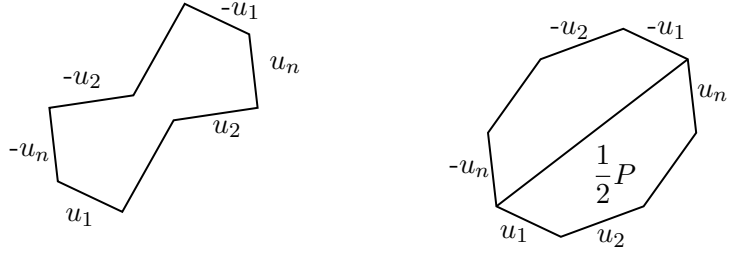

Figure \ref{fig:symorder} suggests that given edges from a set $V$, the best isoperimetric ratio might be achieved by symmetric ordered polygons. This will be confirmed in Theorem \ref{iso thm} below.

Lastly, we introduce rescaling of polygons. For $P=(u_1, u_2, \dots u_{n-1},u_n)$ we set
$ 2P=(u_1,u_1,u_2,u_2,\ldots,u_{n-1},u_{n-1},u_n,u_n)$. 
Similarly we define the polygon $kP$ for every $k\in\bbN$. We will write $ku_i$ and $-ku_i$ in the sequence for $k$ consecutive appearances of $u_i$ and, repsectively, $-u_i$. Rescaling preserves the set $\mathcal{P}_{so}$.
Observe that $l(kP)=kl(P)$ and $a(kP)=k^2a(P)$.
In particular $\gamma(kP)=\gamma(P)$.

\begin{theorem} \label{iso thm}
Given a finite collection of vectors $V\subset \bbZ^d$ with $V=-V$,
there exists $P\in \mathcal{P}_{so}(V)$ such that $\gamma_V=\gamma(P)$. 
\end{theorem}

\begin{lemma} \label{iso lemma}
Let $K<\bbR$ be a subfield, and $d\in \bbN$.
Let $Q\in M_{r\times r}(K)$ be a symmetric matrix with positive coefficients,
and denote by $q$ the corresponding quadratic form.
Let
\[ \Delta=\{x\in \bbR^r~|~\forall i,~ x_i\geq 0,~\sum x_i=1~\}, \]
and $\Delta_q\subset \Delta$ the maximum set of $q$.
Then there exists a finite collection $K$-rational subspaces $V_1,\ldots,V_n<\bbR^r$
such that $\Delta_q=\cup_i (\Delta\cap V_i)$.
\end{lemma}

\begin{proof}[Proof of the lemma]
Denote the boundary of $\Delta$ in its affine span by $\partial \Delta$ and let $(\partial\Delta)_q=\partial \Delta\cap \Delta_q$.
Denote ${\bf 1}=(1,1,\ldots,1)$.
Observe that by Lagrange multiplier theorem,
for $x\in \Delta_q-(\partial\Delta)_q$ there exists some $\lambda\neq 0$ such that $Q(x)=\lambda {\bf 1}$
($\lambda=0$ can occur only for $Q=0$, then the lemma is trivial, so we assume this is not the case).
We consider three cases.

1) $Q$ is invertible. In that case, if there exists $x\in \Delta_q-(\partial\Delta)_q$
then $x$ is the unique solution of $Q(x)=\lambda {\bf 1}$ in the affine span of $\Delta$,
thus $\Delta_q=\{x\}$ and $x$ spans the same line as $Q^{-1}({\bf 1})$, which is $K$-rational.
Otherwise, $\Delta_q\subset \partial \Delta$ and the lemma follows by induction on $d$.

2) There exists $z\in \Ker(Q)$ with $\langle z,{\bf 1}\rangle\neq 0$. Then $\Delta_q \subset \partial\Delta$ and the lemma follows by induction on $r$.
Indeed,
if there exists $x\in \Delta_q-(\partial\Delta)_q$
then $Q(x)=\lambda {\bf 1}$ for some $\lambda\neq 0$, thus we get a contradiction
$0\neq \langle Q(x),z \rangle = \langle x,Q(z)\rangle =0$.

3) $Q$ is not invertible and $\Ker(Q) \perp {\bf 1}$.
Then $\Delta_q=\Delta \cap (\Ker(Q)+(\partial\Delta)_q)$, and the lemma follows again by an induction on $r$.
\end{proof}

\begin{proof}[Proof of the theorem]

Throughout we fix the set $V$ and assume, as we may, that $0\notin V$.
We set $\gamma=\gamma_V$.
First, we show that it is enough to consider the supremum over symmetric ordered polygons.

Since $\gamma$ is invariant under rescaling of polygons, it is enough to consider supremum only over the polygons of even perimeter.  
Let $P_o=(u_1,\ldots,u_{2n})$ be a polygon of even perimeter.
We associate with $P$ two symmetric polygons (see Figure \ref{fig:symmetrizing}):
$$P_+=(u_1,\ldots,u_n,-u_1,\ldots,-u_n) \text{ and }
P_-=(-u_{n+1},\ldots,-u_{2n},u_{n+1},\ldots,u_{2n}).$$

\begin{figure}\label{fig:symmetrizing}

\tikzset{every picture/.style={line width=0.75pt}} 

\begin{tikzpicture}[x=0.75pt,y=0.75pt,yscale=-.8,xscale=.8]

\draw    (35.5,162) -- (35.5,162) -- (162.5,162) ;
\draw [shift={(164.5,162)}, rotate = 180] [color={rgb, 255:red, 0; green, 0; blue, 0 }  ][line width=0.75]    (10.93,-3.29) .. controls (6.95,-1.4) and (3.31,-0.3) .. (0,0) .. controls (3.31,0.3) and (6.95,1.4) .. (10.93,3.29)   ;

\draw    (164.5,162) -- (164.5,162) -- (163.52,80) ;
\draw [shift={(163.5,78)}, rotate = 449.32] [color={rgb, 255:red, 0; green, 0; blue, 0 }  ][line width=0.75]    (10.93,-3.29) .. controls (6.95,-1.4) and (3.31,-0.3) .. (0,0) .. controls (3.31,0.3) and (6.95,1.4) .. (10.93,3.29)   ;

\draw    (163.5,78) -- (163.5,78) -- (136.09,57.21) ;
\draw [shift={(134.5,56)}, rotate = 397.18] [color={rgb, 255:red, 0; green, 0; blue, 0 }  ][line width=0.75]    (10.93,-3.29) .. controls (6.95,-1.4) and (3.31,-0.3) .. (0,0) .. controls (3.31,0.3) and (6.95,1.4) .. (10.93,3.29)   ;

\draw    (134.5,56) -- (36.87,160.54) ;
\draw [shift={(35.5,162)}, rotate = 313.03999999999996] [color={rgb, 255:red, 0; green, 0; blue, 0 }  ][line width=0.75]    (10.93,-3.29) .. controls (6.95,-1.4) and (3.31,-0.3) .. (0,0) .. controls (3.31,0.3) and (6.95,1.4) .. (10.93,3.29)   ;

\draw    (237.5,161) -- (237.5,161) -- (364.5,161) ;
\draw [shift={(366.5,161)}, rotate = 180] [color={rgb, 255:red, 0; green, 0; blue, 0 }  ][line width=0.75]    (10.93,-3.29) .. controls (6.95,-1.4) and (3.31,-0.3) .. (0,0) .. controls (3.31,0.3) and (6.95,1.4) .. (10.93,3.29)   ;

\draw    (366.5,161) -- (366.5,161) -- (365.52,79) ;
\draw [shift={(365.5,77)}, rotate = 449.32] [color={rgb, 255:red, 0; green, 0; blue, 0 }  ][line width=0.75]    (10.93,-3.29) .. controls (6.95,-1.4) and (3.31,-0.3) .. (0,0) .. controls (3.31,0.3) and (6.95,1.4) .. (10.93,3.29)   ;

\draw    (574.5,78) -- (574.5,78) -- (547.09,57.21) ;
\draw [shift={(545.5,56)}, rotate = 397.18] [color={rgb, 255:red, 0; green, 0; blue, 0 }  ][line width=0.75]    (10.93,-3.29) .. controls (6.95,-1.4) and (3.31,-0.3) .. (0,0) .. controls (3.31,0.3) and (6.95,1.4) .. (10.93,3.29)   ;

\draw    (545.5,56) -- (447.87,160.54) ;
\draw [shift={(446.5,162)}, rotate = 313.03999999999996] [color={rgb, 255:red, 0; green, 0; blue, 0 }  ][line width=0.75]    (10.93,-3.29) .. controls (6.95,-1.4) and (3.31,-0.3) .. (0,0) .. controls (3.31,0.3) and (6.95,1.4) .. (10.93,3.29)   ;

\draw    (446.5,162) -- (473.91,182.79) ;
\draw [shift={(475.5,184)}, rotate = 217.18] [color={rgb, 255:red, 0; green, 0; blue, 0 }  ][line width=0.75]    (10.93,-3.29) .. controls (6.95,-1.4) and (3.31,-0.3) .. (0,0) .. controls (3.31,0.3) and (6.95,1.4) .. (10.93,3.29)   ;

\draw    (475.5,184) -- (573.13,79.46) ;
\draw [shift={(574.5,78)}, rotate = 493.04] [color={rgb, 255:red, 0; green, 0; blue, 0 }  ][line width=0.75]    (10.93,-3.29) .. controls (6.95,-1.4) and (3.31,-0.3) .. (1,0) .. controls (3.31,0.3) and (6.95,1.4) .. (10.93,3.29)   ;

\draw    (365.5,77) -- (238.5,77) ;
\draw [shift={(236.5,77)}, rotate = 360] [color={rgb, 255:red, 0; green, 0; blue, 0 }  ][line width=0.75]    (10.93,-3.29) .. controls (6.95,-1.4) and (3.31,-0.3) .. (0,0) .. controls (3.31,0.3) and (6.95,1.4) .. (10.93,3.29)   ;

\draw   (236.5,77) -- (237.48,159) ;
\draw [shift={(237.5,161)}, rotate = 269.32] [color={rgb, 255:red, 0; green, 0; blue, 0 }  ][line width=0.75]    (10.93,-3.29) .. controls (6.95,-1.4) and (3.31,-0.3) .. (0,0) .. controls (3.31,0.3) and (6.95,1.4) .. (10.93,3.29)   ;

\draw [dash pattern={on 0.84pt off 2.51pt}]   (163.5,78) -- (35.5,162) ;

\draw [dash pattern={on 0.84pt off 2.51pt}]   (237.5,161) -- (365.5,77) ;

\draw  [dash pattern={on 0.84pt off 2.51pt}]  (574.5,78) -- (446.5,162) ;

\draw (100,176) node  [align=left] {$\displaystyle u_{1}$};
\draw (180,121) node  [align=left] {$\displaystyle u_{2}$};
\draw (162,55) node  [align=left] {$\displaystyle u_{3}$};
\draw (75,91) node  [align=left] {$\displaystyle u_{4}$};
\draw (302,175) node  [align=left] {$\displaystyle u_{1}$};
\draw (391,120) node  [align=left] {$\displaystyle u_{2}$};
\draw (573,55) node  [align=left] {$\displaystyle u_{3}$};
\draw (486,91) node  [align=left] {$\displaystyle u_{4}$};
\draw (221,120) node  [align=left] {\mbox{-}$\displaystyle u_{2}$};
\draw (295,62) node  [align=left] {$\displaystyle -u_{1}$};
\draw (547,141) node  [align=left] {$\displaystyle -u_{4}$};
\draw (441,180) node  [align=left] {$\displaystyle -u_{3}$};
\draw (110,208) node  [align=left] {$\displaystyle P=(u_1,u_2,u_3,u_4)$};
\draw (306,209) node  [align=left] {$\displaystyle P_{+}$};
\draw (521,209) node  [align=left] {$\displaystyle P_{-}$};
\end{tikzpicture}
\caption{Constructing symmetric $P_+$ and $P_-$ from $P$.}
    
\end{figure}
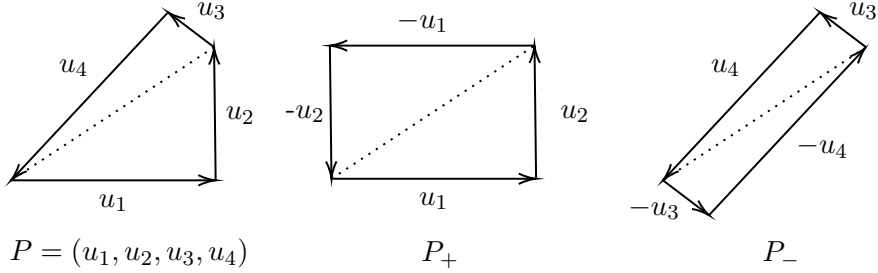

Observe that  $l(P)=l(P_+)=l(P_-)$ and $a(P)=a\left(\half P_-\right)+a\left(\half P_+\right)$ (by the fact that $u_1+\cdots+u_n=u_{n+1}+\cdots+u_{2n}$). Hence,
$\gamma(P)=\frac{1}{2}(\gamma(P_-)+\gamma(P_+))$.

We get for every $P\in \mathcal{P}$,
\[ \gamma(P)=\gamma(2P) \leq \max\{\gamma((2P)_-),\gamma((2P)_+)\}. \]

Clearly, one can associate with any symmetric polygon $P=(u_1 \dots u_{2n})$ a symmetric ordered one $P_o$. This is done by rearranging the vectors $u_i$, for $1\leq i\leq n$, in increasing order (with respect to $\prec$) and by doing a similar rearrangement of $u_i$, for $n+1 \leq i \leq 2n$,  to keep the polygon symmetric. In this case
$l(P_o)=l(P)$ and $a(P_o)\geq a(P)$.
From geometric point of view, this means that a polygon of largest area with given sides is the convex one (recall, that for symmetric polygons, it is enough to consider the area of $\half P$, which is convex when $P$ is ordered). 
In particular, we get that for every symmetric $P\in\mathcal{P}(V)$, $\gamma(P_o)\geq \gamma(P)$.

The area $a(P)$ does not depend on the cyclic permutation of vectors in the polygon, hence we can restrict our attention to
$$\mathcal{P}_{so}^m := \left\{ P=(u_1,\ldots,u_n,-u_1, \ldots,-u_n) \in \mathcal{P}_{so} ~|~ u_1 \prec u_i , \text{ for all }1\leq i\leq n  \right\}$$
Therefore,
\[ \gamma=\sup \{\gamma(P)~|~P\in \mathcal{P}_{so}^m\}.\]

We proceed to show that this supremum is attained. Let 
\[V=\{ v_1, \dots, v_r, -v_1, \dots, -v_r \}.\]

With a symmetric ordered polygon $P\in \mathcal{P}_{so}^m$ we associate $(c_i)\in \bbZ^{2r}_+$, where for each $1\leq i\leq 2r$, $c_i$ counts the number of appearances of $v_i \in V$ in the first half of the edges of $P$ (mind, that the remaining half are just the inverses, hence are completely determined) . Since $(c_i)$ completely determines $P$, the set $\mathcal{P}_{so}^m$ is in a bijection with $\bbZ^{2r}_+$. Note that $a$ is a quadratic form on $\bbR^{2r}$, defined by the symmetric and rational matrix $Q=\left( \det(v_i,v_j) \right)$ (see equation \ref{eq:areasymmetric}). We apply Lemma~\ref{iso lemma}. Since any maximum of $a$ on the simplex $\Delta$ is obtained at a rational vector, we deduce by rescaling that $\gamma$ attains maximum in $\bbZ^{2r}_+$. The bijection described above produces the polygon $P\in P_{so}$ satisfying the theorem.
\end{proof}

\begin{remark}
In fact it is an easy consequence of the 2-dimensional Brunn-Minkowski theorem that the polygon achieving the maximum for $\gamma$ is unique up to a homothety (and cyclic permutation).
\end{remark}

Next, we describe further polygons given by Theorem~\ref{iso thm}, which achieves the maximal isoperimetric constant, by specifying which vectors are used in them. This will be cruicial for studying the norm on the Heisenberg group in Section \ref{sec5}.

\begin{prop} \label{genset}
Let $V\subset \bbZ^2$ be a finite set of vectors, $V=-V$. Let $P_0$ be a symmetric ordered polygon satisfying Theorem~\ref{iso thm}, denote by $V'$ the vectors used in $P_0$, then $V'=ext(conv(V))$.
\end{prop}

\begin{proof}
First we show that $V'\subset ext(conv(V))$. Recall that $P_0$ is a $V$-polygon maximizing the ratio $\gamma(P_0)$ between the enclosed area and the square of its combinatorial length. If $v\in V',v \notin ext(conv(V))$  we will show that it does not appear in $P_0$. Suppose it does. Then there exists $k\geq 1$ such that $kv \in \partial conv(V)$ and $kv=tv_1 +(1-t)v_2$ for some $v_1,v_2 \in ext(conv(V))$ and $0\leq t \leq 1$. Since $V\subset \bbZ^2$, we have $k,t \in \bbQ$. Write $k=k'/n,t=t'/n$, where $k',t',n\in \bbN$. By assumption, in $k'P_0$ there is an appearance of $k'v$. Let $P_1$ be the polygon obtained from $k'P_0$ by replacing $\pm k'v$ with $\pm t'v_1$ and $\pm(n-t')v_2$ in such a way, so that $P_1$ is a symmetric ordered polygon. 

\begin{figure}\label{fig:increasing_area}

\tikzset{every picture/.style={line width=0.75pt}} 

\begin{tikzpicture}[x=0.75pt,y=0.75pt,yscale=-.7,xscale=.85]

\draw    (121.5,269) -- (221.5,252) ;

\draw    (259.5,124) -- (221.5,252) ;

\draw    (193.5,53) -- (259.5,124) ;

\draw  [dash pattern={on 0.84pt off 2.51pt}]  (236.5,204) -- (267.5,191) ;

\draw  [dash pattern={on 0.84pt off 2.51pt}]  (267.5,191) -- (249.5,163) ;

\draw    (93.5,70) -- (55.5,198) ;

\draw    (93.5,70) -- (193.5,53) ;

\draw    (55.5,198) -- (121.5,269) ;

\draw  [dash pattern={on 0.84pt off 2.51pt}]  (71.5,149) -- (47.5,121) ;

\draw  [dash pattern={on 0.84pt off 2.51pt}]  (53.5,121) -- (81.5,108) ;

\draw    (552.5,151) -- (575.5,166) -- (541.5,254) -- (525.5,261) -- (398.5,215) -- (375.5,200) -- (409.5,112) -- (425.5,105) -- cycle ;

\draw    (475.5,183) -- (550.65,151.77) ;
\draw [shift={(552.5,151)}, rotate = 517.4300000000001] [color={rgb, 255:red, 0; green, 0; blue, 0 }  ][line width=0.75]    (10.93,-3.29) .. controls (6.95,-1.4) and (3.31,-0.3) .. (0,0) .. controls (3.31,0.3) and (6.95,1.4) .. (10.93,3.29)   ;

\draw    (475.5,183) -- (426.58,106.68) ;
\draw [shift={(425.5,105)}, rotate = 417.34000000000003] [color={rgb, 255:red, 0; green, 0; blue, 0 }  ][line width=0.75]    (10.93,-3.29) .. controls (6.95,-1.4) and (3.31,-0.3) .. (0,0) .. controls (3.31,0.3) and (6.95,1.4) .. (10.93,3.29)   ;

\draw    (475.5,183) -- (487.87,145.9) ;
\draw [shift={(488.5,144)}, rotate = 468.43] [color={rgb, 255:red, 0; green, 0; blue, 0 }  ][line width=0.75]    (10.93,-3.29) .. controls (6.95,-1.4) and (3.31,-0.3) .. (0,0) .. controls (3.31,0.3) and (6.95,1.4) .. (10.93,3.29)   ;

\draw (229,176) node [scale=2,rotate=-20.5] [align=left] {\{};
\draw (207,173) node  [align=left] {$\displaystyle k'v$};
\draw (257,214) node  [align=left] {$\displaystyle t'v_{1}$};
\draw (296,163) node  [align=left] {$\displaystyle ( n-t') v_{2}$};
\draw (153,156) node  [align=left] {$\displaystyle k'P_{0}$};
\draw (530,175) node  [align=left] {$\displaystyle v_{1}$};
\draw (431,138) node  [align=left] {$\displaystyle v_{2}$};
\draw (472,151) node  [align=left] {$\displaystyle v$};
\draw (485,208) node  [align=left] {$\displaystyle conv\left(V\right)$};

\end{tikzpicture}
\caption{Replacing $ k'v$ by $t'v_{1}$ and $( n-t') v_{2}$ in $k'P_0$ increases the area.}
    
\end{figure}
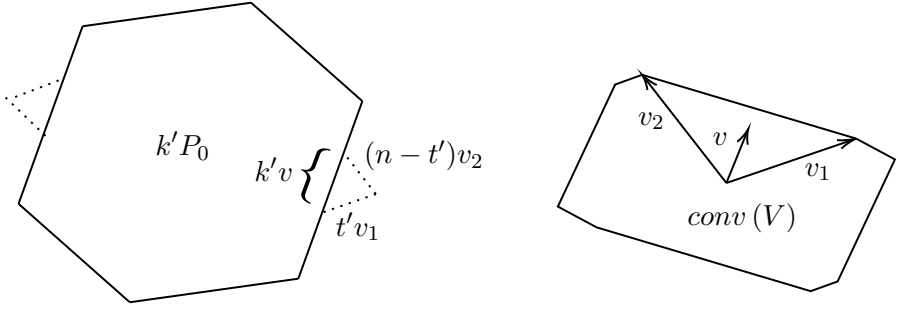

Note that $l(P_1)\leq l(P_0)$, since $k'\geq n$. However, $a(P_1)>a(k'P_0)$ (see Figure \ref{fig:increasing_area}). Hence $\gamma(P_1)>\gamma(k'P_0)=\gamma(P_0)$, contradicting that $\gamma$ attains its maximum at $P_0$. 

For other inclusion, let $v\in ext(conv(V))$ and suppose that $v\notin V'$. We use the order $\prec$ on $V$ defined in the beginning of this section. When restricted to $ext(conv(V))$ this order is a strict total order.

Let $v_1=\max \{ v_i\in V' ~|~ v_1 \prec v \}$ and $v_2=\min \{ v_i \in V' ~|~ v_2 \succ v \}$. 
Similarly to the previous paragraph, there exist integers $k,t,n>0$, such that $k,t<n$ and $tv_1 +(n-t)v_2 = k v$. Let $m>n$ be large (specified later). Construct a symmetric ordered polygon $P_1$ by replacing $\pm tv_1$ and $\pm (n-t)v_2$ in $mP_0$ with $\pm nv$ in the appropriate places. Note that $l(P_0)=l(P_1)$. We are left to argue that $a(P_1)>a(P_0)$, which will conradict that $P_0$ attains the maximum of $\gamma$. 

Up to cyclic permutation we write
$$mP_0=(m_1 v_1, m_2v_2, \dots m_{r'}v_{r'}, -m_1v_1, -m_2v_2, \dots, -m_{r'} v_{r'}),$$
with $m_i \geq m$ for each $i$. Then, 
$$P_1=((m_1-t) v_1, nv, (m_2-n+t)v_2, \dots m_{r'}v_{r'}, -(m_1-t) v_1, -nv, \dots, -m_{r'} v_{r'})).$$ 
We compare the areas:
\[
\begin{array}{ll}
a(P_1)-a(P_0) =  
\\
     \quad \quad  {\displaystyle \sum_{3\leq j \leq {r'}} }(n \det(v,v_j) - t\det(v_1,v_j) - (n-t)\det(v_2,v_j))  \\
     \quad \quad +  n(m_1-t)\det( v_1, v ) + n(m_2-n+t)\det(v,v_2)  \\
     \quad \quad + (m_1-t)(m_2-n+t)\det(v_1,v_2)  - m_1m_2\det(v_1,v_2) .
\end{array}     
\]

 We have $k \det(v,w) = t\det(v_1,w) + (n-t)\det(v_2,w)$ for any $w\in \bbR^2$, because $kv=tv_1+(n-t)v_2$. Also, $n\geq k$ and $\det(v,v_j)\geq 0$ for all $3\leq j\leq r'$, therefore, all the summands inside the sum in the last equation are non-negative. Finally, we claim that the sum of the remaining terms is positive if $m$ and, hence, $m_1,m_2$ are large enough. Indeed, it represents the difference between the areas of the quadrilateral and the triangle in Figure \ref{fig:area_calc}.
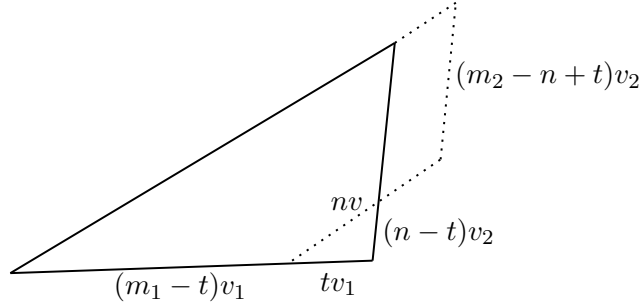
\begin{figure}\label{fig:area_calc}

\tikzset{every picture/.style={line width=0.75pt}} 

\begin{tikzpicture}[x=0.75pt,y=0.75pt,yscale=-.7,xscale=.7]

\draw    (64.5,252) -- (325.5,243) ;

\draw    (341.5,86) -- (325.5,243) ;

\draw  [dash pattern={on 0.84pt off 2.51pt}]  (267.5,243) -- (374.5,170) ;

\draw  [dash pattern={on 0.84pt off 2.51pt}]  (374.5,170) -- (385.5,57) ;

\draw    (64.5,252) -- (341.5,86) ;

\draw  [dash pattern={on 0.84pt off 2.51pt}]  (385.5,57) -- (341.5,86) ;

\draw (187,262) node  [align=left] {$\displaystyle ( m_{1} -t) v_{1}$};
\draw (452,110) node  [align=left] {$\displaystyle ( m_{2} -n+t) v_{2}$};
\draw (301,260) node  [align=left] {$\displaystyle tv_{1}$};
\draw (373,223) node  [align=left] {$\displaystyle ( n-t) v_{2}$};
\draw (307,202) node  [align=left] {$\displaystyle nv$};


\end{tikzpicture}
\caption{Replacing $tv_{1}$ and $( n-t) v_{2}$ by $nv$ gains more area than it loses.}
    
\end{figure}
Clearly, the area of the quadrilateral (grows linearly in $m_2$) is bigger than the area of the triangle (fixed) when $m$ is large enough, which finishes the proof.
\\
\end{proof}

\section{The Heisenberg Group - Proof of Theorem B.}\label{sec5}

In this section $G$ denotes the discrete Heisenberg group $H_3(\bbZ)$,
that is $G$ is in bijection with $\bbZ^3$ as a set and the multiplication in $G$ is given by
\[ (x_{1},y_{1},z_{1})(x_{2},y_{2},z_{2})=(x_{1}+x_{2},y_{1}+y_{2},z_{1}+z_{2}+x_{1}y_{2}). \]
Observe that the center $Z<G$ consists of the elements of the form $(0,0,z)$ and that $G/Z$ is naturally isomorphic to $\bbZ^2$.
In particular $G'=Z$.
Denote the abelianization map $\pi:G\to \bbZ^2$.
For a given generating set $S\subset G$ we denote
$\bar{S}=\pi(S)$.
This is a generating set for $\bbZ^2$.
We denote
by
$|\cdot|_S$ the group norm on $G$ and by $|\cdot|_{\bar{S}}$ the corresponding group norm on $\bbZ^2$.

\begin{prop} \label{Heis}
Let $S\subset G$ be a finite symmetric generating
set for $G=H_{3}(\mathbb{Z})$.
Then there exist constants $L,C\geq 0$ such that
for every $(x,y,z)\in G$,
\[ |z|\geq L\max\left\{ x^{4},y^{4}\right\} \quad \Rightarrow \quad
\Big| |(x,y,z)|_{S}-\left(\sqrt{\frac{|z|}{\gamma}}-|(x,y)|_{\bar{S}}\right) \Big| \leq C, \]
where $\gamma=\gamma_{\bar{S}}$ is the isoperimetric constant given in Theorem~\ref{iso thm}.
\end{prop}

\begin{cor} \label{Heis is EH}
The discrete Heisenberg group satisfies EH.
\end{cor}

In particular, Corollary~\ref{Heis is EH} implies Theorem B via Proposition~\ref{EH}.

\begin{proof}[Proof of Proposition~\ref{Heis} $\Rightarrow$ Corollary~\ref{Heis is EH}]
Let $S\subset G$ be a finite symmetric generating
set.
We will show that $S$ satisfies EH.
Let $L,C$ and $\gamma$ be as in Proposition~\ref{Heis}.
Let $M_x=\max\{x~|~(x,y,z)\in S\}$ and $M_y=\max \{y~|~(x,y,z)\in S\}$,
and let $R=\max\{|(0,0,z)|~\big|~|z|\leq \max\{M_x,M_y\}\}$.
We set
\[ D=\max\left\{\frac{1}{2\sqrt{\gamma L}z_0}+2C,R+2|(0,1,0)| , R + 2|(1,0,0)| \right\}. \]
Let $g_0\in G'$.
$G'=Z$ so $g_0=(0,0,z_0)$ for some $z_0$.
We will assume $z_0>0$, as $S$ is symmetric.
Consider the set
\[ B=\{(x,y,z)~\big|~|x|,|y|<z_0,~|z|<Lz_0^4+z_0 \}. \]
This is a finite set, so there exists an $n$ such that
\[ \mbox{for all } w\in G, \quad\quad |w|_S>n \quad \Rightarrow \quad w\notin B. \]
We will be done by showing
\[ \mbox{for all } w\in G, \quad\quad w\notin B \quad \Rightarrow \quad \big| |g_0w|_S-|w|_S \big|\leq D. \]
Fix $w=(x,y,z) \notin B$.
Assume $|z|\geq Lz_0^4+z_0$.
Then, both $w$ and $g_0w$ have last coordinate $\geq Lz_0^4$ and by Proposition~\ref{Heis},

\[  
\begin{array}{rll}
 \big| |g_0w|-|w| \big| & =\big| |(x,y,z+z_0)|-|(x,y,z)| \big| & \\
 &\leq \left|\left(\sqrt{\frac{|z|+z_0}{\gamma}}-|(x,y)|_{\bar{S}}\right)-\left(\sqrt{\frac{|z|}{\gamma}}-|(x,y)|_{\bar{S}}\right) \right| +2C  \\
& = \left|\sqrt{\frac{|z|+z_0}{\gamma}}-\sqrt{\frac{|z|}{\gamma}} \right| +2C  \leq
\left|\frac{z_0}{\sqrt{\gamma}\left(\sqrt{|z|+z_0}+\sqrt{|z|}\right)} \right| +2C  \\
& \leq \left|\frac{z_0}{\sqrt{\gamma}\left(2\sqrt{ Lz_0^4}\right)} \right| +2C  =\frac{1}{2\sqrt{\gamma L}z_0}+2C  \leq D.
\end{array}
\]

Otherwise, $z<Lz_0^4+z_0$ and since $(x,y,z)\notin B$ we must have $|x|\geq z_0$ or $|y| \geq z_0$.
Assume $|x|\geq z_0$.
By considering an $S$-geodesic from $e$ to $w$, we can find words
$w_1=(x_1,y_1,z_1)$ and $w_2=(x_2,y_2,z_2)$ such that
$w=w_1w_2$, $|w|=|w_1|+|w_2|$ and $\big| z_0-|x_1| \big|\leq M_x$.
Check that
\[ (0,\pm1,0)(x_1,y_1,z_1)(0,\mp1,0)=(x_1,y_1,z_1\pm x_1). \]
Then we have 
\[ g_0w=
\left\{
\begin{array}{ll}
(0,0,z_0-x_1)(0,+1,0)w_1(0,-1,0)w_2 & \mbox{for } x_1\geq z_0 \\
(0,0,z_0+x_1)(0,-1,0)w_1(0,+1,0)w_2 & \mbox{for } -x_1\geq z_0
\end{array}\right.
\]
and in any case
\[ \big| |g_0w|-|w| \big| \leq \left|\left(0,0,| z_0-|x_1||\right)\right|+2|(0,1,0)| \leq R+ 2|(0,1,0)|\leq D. \]
The case $|y|\geq z_0$ is similar,
and completes the proof.
\end{proof}

The rest of the section is devoted to the proof of Proposition~\ref{Heis}.

\begin{lemma}\label{area-z}
Let $S$ be a finite symmetric set of generators for $G=H_3(\bbZ)$. 
Let $w$ be a word of length $n$ in the free group generated by $S$
which has image $(0,0,z)\in G$. There exists $K=K(S)$, such that if we  $a(w)$ denotes the signed Euclidean area of the corresponding polygon obtained in $\bbZ^2$, then 
$$|a(w)-z|\leq Kn.$$
\end{lemma}

\begin{proof}
Let $w=(s_1s_2\cdots s_n)$, $s_i=(x_i,y_i,z_i)\in S$.
Then we have
\[ a(w)=\half \sum_{i<j} (x_iy_j-x_jy_i) \quad \mbox{and} \quad 
z=\sum_{i<j} x_iy_j +\sum_i z_i. \] 
We consider the word $w^{-1}=(s_n^{-1}\cdots s_1^{-1})$
and compute its $z$-coordinate. Using $s_i^{-1}=(-x_i,-y_i,x_iy_i-z_i)$ we get
\[ -z= \sum_{i> j} (-x_i)(-y_j) +\sum_i  (x_iy_i-z_i). \]
Thus we get 
\[ z=\half(z-(-z))= \half\sum_{i<j} (x_iy_j-x_jy_i) +\sum_i (z_i-x_iy_i) \] 
and
\[  z-a(w)=\sum_i z_i-x_iy_i.\]
The lemma follows by setting $K=\max\{|z-xy|: (x,y,z)\in S\}$.

\end{proof}

\begin{prop} \label{hight}
Let $S\subset G$ be a finite symmetric generating
set for $G=H_{3}(\mathbb{Z})$.
Then there exists a constant $C\geq 0$ such that
for every $(0,0,z)\in G$,
\[ \left||(0,0,z)|  - \frac{\sqrt{z}}{\gamma} \right| \leq C\]
where $\gamma=\gamma_{\bar{S}}$ is the isoperimetric constant given in Theorem~\ref{iso thm}.

\end{prop}

\begin{proof}
Assume without loss of generality that $z\geq 0$.
For any $w$ that represents $(0,0,z)$ we have
$l(w)^2\gamma \geq |a(w)| \geq z-Kl(w)$ (with $K$ from Lemma \ref{area-z}),
hence 
\[ (l(w)+K/2)^2 \geq z/\gamma \quad \mbox{and} \quad 
l(w) \geq \sqrt{z/\gamma}-K/2. \]
By Theorem~\ref{iso thm}
we know that there exists a polygon $P_0$ such that 
for every $k\in \bbN$, $a(kP_0)/l(kP_0)^2=\gamma$.
Denote $z_0=l(P_0)$.
We fix $z_0-1$ words $w_1,\ldots,w_{z_0}$ representing the elements $(0,0,1),\ldots,(0,0,z_0)$ correspondingly,
and set $M=\max\{l(w_r)~|~r=1,\ldots,z_0\}$.
We write $z=kz_0+r$ for some $k\in\bbN$ and $1\leq r\leq z_0$.
We consider the word $w'$ representing the polygon $kP_0$.
Then
\[ l(w')^2 = a(w')/\gamma \leq (kz_0+Kl(w'))/\gamma. \]
Thus
\[ (l(w')-K/2\gamma)^2 \leq kz_0/\gamma + (K/2\gamma)^2. \]
Since $kz_0<z$ we get
\[ l(w') \leq \sqrt{z/\gamma} +K/\gamma, \]
and, since $w=w'w_r$, 
\[ l(w)\leq l(w')+M\leq  \sqrt{z/\gamma}+M+K/\gamma. \] 

\end{proof}

\begin{prop} \label{heightcontrol}
There exists $E>0$, such that for any word $w$, we have
$ |h(w)|\leq El(w)^{2}$.
\end{prop}

\begin{proof}

Let $w$ be a word to $(x,y,h(w))$. If $(x,y)=(0,0)$ then by definition of $\gamma$ we have 
\[|h(w)|\leq l(w)^{2}/\sqrt{\gamma},\]
and we are done. 

If $(x,y)\neq(0,0)$, let $w'=(-x,0,0)(0,-y,0)$. The length of $w'$ is bounded by
\[\begin{array}{ll}
    l(w') & \leq |x|\cdot |(1,0,0)|+|y|\cdot |(0,1,0)| \\
     & \leq 2 \max \{|x|,|y|\} \max\{|(0,1,0)|,|(1,0,0)|\}.
\end{array}\]

Let $m_x=\max\{|x|,(x,y,z)\in S\}, m_y=\max\{|y|,(x,y,z)\in S\}$. 
\[ l(w) \geq \frac{\max\{|x|,|y|\}}{\max\{m_x,m_y\}}.\]

Therefore, for $D=2 \max\{m_x,m_y\}\max\{|(0,1,0)|,|(1,0,0)|\}$ we have 
\[l(w')\leq Dl(w).\]

Note that $ww'$ is a word representing $(0,0,h(w))$, hence  
\[ |h(w)|=|h(ww')|\leq (D+1)^2l(w)^{2}/\sqrt{\gamma}, \]
and we are done. \\

\end{proof}

Now we are ready to prove proposition ~\ref{Heis}.

\begin{proof}
Let $(x,y,z)\in G$ as in proposition. Let $\bar{w}$ be a geodesic word in $\mathbb{Z}^{2}$ from the origin to $(-x,-y)$ in the generators $S'=ext(conv(\bar{S}))$. We can assume $(x,y)\in span\{S'\}$, i.e. the geodesic word is of form $(-x,-y)=a_i \bar{s_i} + a_j \bar{s_j}$, where $\bar{s_i},\bar{s_j} \in S'$ are two adjacent generators, i.e. they share a face in $conv(\bar{S})$. 
Clearly, $|(-x,-y)|_{\bar{S}}=a_i +a_j$. 

Let $w=s_i^{a_i}s_j^{a_j}$ be a lift of $\bar{w}$ to the Heisenberg group. 
The word $w$ represents $(-x,-y,h(w))$ and $|(-x,-y,h(w))|_S=|(-x,-y)|_{\bar{S}}$. The length $l(w) = a_i+a_j \leq 2 \max\{|x|,|y|\}$. From Proposition~\ref{heightcontrol} there exists $E\geq 0$ such that $h(w)\leq E l(w)^2\leq 2E \max\{x^2,y^2\}$.  

By the triangle inequality
\[|(x,y,z)|+|(-x,-y,h(w))| \geq |(0,0,z-xy+h(w))|. \]

For $L\geq 4E^2+4E+1 $  we get
 \[
 \begin{array}{ll}
     (h(w)-xy)^2 & \leq  h(w)^2 +2xy h(w) +x^2y^2  \\
      &\leq  4E^2 \max\{x^4,y^4\} +4E\max\{x^4,y^4\} + \max\{x^4,y^4\}  \leq z.
 \end{array}  \]
hence
\[
    \left| \sqrt{z-xy+h(w)} - \sqrt{z} \right| \leq 1 .
\]
Therefore, by Proposition \ref{hight} we obtain
\[ |(x,y,z)| \geq \sqrt{z/\gamma} - |(x,y)|_{\bar{S}} + C+1. \]

For the other direction in the last inequality, we will construct a word to $(x,y,z)$ of the needed length.
Let $w=s_i^{a_i}s_j^{a_j}$ be, as before, a word to $(-x,-y,h(w))$. Let $w'$ be a word to $(0,0,z-h(w))$ obtained as in Proposition~\ref{hight}, namely it is given by a multiple of $P_0$ multiplied by some bounded commuting factor $w_r$.    

From the argument as  before  we have $\big| \; |(0,0,z-h(w)|-|(0,0,z)| \; \big|\leq 1$.

For $L \geq 2 a(P_0)$ we get
\[ \frac{z}{a(P_0)} \geq 2\max\{ |x|^4,|y|^4 \} \geq \max 2\{a_i,a_j\}. \]
By Proposition \ref{genset} all elements from $S'$ appear in $P_0$ and, hence, their lifts appear in $w'$. Moreover, since $\bar{s_i}$ and $\bar{s_j}$ are adjacent in $P_0$ and by the last inequality we know that a subword $w=s_i^{a_i}s_j^{a_j}$ appears in $w'$. Since $h(w')$ doesn't depend on cyclic permutation of letters, we can assume that $w$ appears as suffix in $w'$. Therefore, $w'w^{-1}$ has cancellation and is a word to $(x,y,h(w)+h(w'))=(x,y,z)$ of length $\leq \sqrt{z/\gamma} - |(x,y)|_{\bar{S}} + C+1 $.

\end{proof}

\newpage

\bibliographystyle{plain}
\bibliography{horofunctions}

\end{document}